\documentclass{proc-l}
\usepackage{amsmath}
\usepackage{amssymb}
\usepackage{amsthm}
\usepackage{graphicx}

\newtheorem{definition}{Definition}
\newtheorem{theorem}{Theorem}
\newtheorem{lemma}{Lemma}
\newtheorem{corollary}{Corollary}

\begin{document}
\title{SO(2)-Congruent Projections of Convex Bodies with Rotation About the Origin}

\author{Benjamin Mackey}
\thanks{This research as supported in part by the NSF Grant, DMS-1101636. I would also like to thank Dmitry Ryabogin for helping me organize and revise this paper.}
\address{Department of Mathematics, Michigan State University, East Lansing, MI, 48824.}
\email{mackeybe@msu.edu}

\begin{abstract}
We prove that if two convex bodies $ K, L \subset \mathbb{R}^3$ satisfy the property that the orthogonal projections of $K$ and $L$ onto every plane containing the origin are roations of each other, then either $K$ and $L$ coincide or $L$ is the image of $K$ under a reflection about the origin.
\end{abstract}


\date{}

\maketitle

\section{Introduction}

In this paper, we will prove the following theorem:

\begin{theorem} \label{thm: Main}

Let $K,L \subset \mathbb{R}^3$ be convex bodies containing the origin as an interior point such that for every $\xi \in S^2$, the projection $K_{|\xi^{\perp}}$ can be rotated about the origin into $L_{|\xi^{\perp}}$. Then either $K=L$ or $K$ can be obtained by reflecting $L$ about the origin.

\end{theorem}

Several related results have been proven under various assumptions about the type of congruence the projections satisfy. Wilhelm S\"uss proved that if each projection $K_{|\xi^{\perp}}$ is some parallel translation of $L_{|\xi^{\perp}}$, then $K$ and $L$ are parallel (\cite{Golubyatnikov} page 8). Vladimir Golubyatnikov allowed for both shifts and rotations, and proved two different theorems with variations on the symmetry of the projections (\cite{Golubyatnikov} page 13) and smoothness of the bodies (\cite{Golubyatnikov} page 22). This new result differs in that no symmetry assumptions are made about the bodies, but the freedom to translate projections is lost.

\section{Notation and Definitions}

Throughout this paper, $\mathbb{R}^n$ will refer to the $n$-dimensional Euclidean space, and $S^{n-1} = \{\xi \in \mathbb{R}^n : |x| =1 \}$ will denote the unit sphere. The Euclidean inner product of two vectors $x,y \in \mathbb{R}^n$ will be denoted by $x \cdot y.$ For a unit vector $\xi \in S^{n-1}$, the hyperplane orthogonal to $\xi$ is denoted by $\xi^{\perp}= \{ x \in \mathbb{R}^n : x \perp \xi \}$. The set $\xi^{\perp} \cap S^2$ is the great circle of unit vectors orthogonal to $\xi$. If $K \subset \mathbb{R}^3$ is a convex body containing the origin and $\xi \in S^2$, the section of $K$ orthogonal to $\xi$ is the set $K \cap \xi^{\perp}$. Given $E \subset S^2$ endowed with the spherical metric,  the interior of $E$ will be denoted by $\text{int}(E)$ and the closure of $E$ will be denoted by $\overline{E}$.

Let $\xi \in S^{2}$, $r$ be some nonnegative number, and let $x \in \mathbb{R}^3$. Then the image of $x$ rotated by an angle of $r\pi$ about the linear subspace spanned by $\xi$ will be denoted by $R_{\xi,r}(x)$. Given $\epsilon > 0,$ the spherical disk of radius $\epsilon \pi$ centered at $\xi$ in $S^2$ will be called $S(\xi, \epsilon)$. We recall some standard concepts in the study of convexity:

\begin{definition}
Let $\xi \in S^{n-1}$ and $K \subset \mathbb{R}^n$ be a convex body. The \textbf{orthogonal projection} of $K$ in the direction $\xi$ is the set $K_{|\xi^{\perp}} = \{y \in \xi^{\perp}: \exists \lambda \in \mathbb{R}, y+ \lambda \xi \in K\}$.
\end{definition}

\begin{definition}
Let $K \subset \mathbb{R}^n$ be a convex body. The \textbf{support function} of $K$ is the map $h_K: S^{n-1} \mapsto \mathbb{R}$ defined by $h_K(\xi)=\max\{u \cdot \xi: u \in K\}$. The \textbf{width function} of $K$ is defined by $width_K(\xi)=(h_K(\xi)+h_K(-\xi))/2$. A body $K$ has \textbf{constant width} if $width_K$ is a constant function on $S^{n-1}$.
\end{definition}

\begin{definition}
If $K \subset \mathbb{R}^n$ is a convex body, the \textbf{polar dual} of $K$ is the set $K^* = \{  x \in \mathbb{R}^n : x \cdot y \leq 1, \forall  y \in K \}$.
\end{definition}

\begin{definition}
Let $K \subset \mathbb{R}^n$ be a convex body containing the origin. The \textbf{radial function} $\rho_K: S^{n-1} \mapsto \mathbb{R}$ is defined by $\rho_K(\xi) = \max \{ \lambda \in \mathbb{R}: \lambda\xi \in K \}$.
\end{definition}

\section{Auxillary Results}

For any $r \in \mathbb{R},$ define $F_r \subset S^2$ by ${F_r}= \{ \xi \in S^2: K_{|\xi^{\perp}}$ rotated by $r\pi$ is $L_{|\xi^{\perp}}\}$. Observe that if a rotation of magnitude between $\pi$ and $2\pi$ in the clockwise direction is necessary for $K_{|\xi^{\perp}}$ to coincide with $L_{|\xi^{\perp}}$, then a rotation of less than $\pi$ in the counterclockwise direction makes the projections coincide. Thus, only $r \in [0,1]$ need be considered. It follows that $\xi \in F_0$ if and only if $K_{|\xi^{\perp}}=L_{|\xi^{\perp}}$, and $\xi \in F_1$ if and only if the image of $K_{|\xi^{\perp}}$ under a reflection about the origin is $L_{|\xi^{\perp}}$. The conclusion of Theorem \ref{thm: Main} can be rewritten as ``either $S^2=F_0$ or $S^2=F_1$."

\begin{lemma} \label{lem: closed}

For all $r \in [0,1]$, the set ${F_r}$ is closed.

\end{lemma}

\begin{proof}

If $F_r$ is empty, then it is trivially closed. If $F_r$ is nonempty, let $\xi_n$ be a sequence in $F_r$, and suppose $\xi_n$ converges to $\xi \in S^2.$ Let $\theta \in \xi^{\perp}$ be arbitrary. For each $n$, pick some $\theta_n \in \xi_n^{\perp}$ so that $\theta_n$ converges to $\theta.$ Since each $\xi_n \in F_r$, we have $h_{L}(R_{\xi_n,r}(\theta_n))=h_{K}(\theta_n)$ for each $n$.

By Rodrigues' rotation formula (\cite{Koks} page 147), 
$$R_{\xi_n,r}(\theta_n)= \theta_n \cos(r\pi)+(\xi_n \times \theta_n)\sin(r\pi)+ \xi_n(\xi_n \cdot \theta_n)(1-\cos(r\pi)).$$ Taking the limit as $n$ approaches infinity, we see that $R_{\xi_n, r}(\theta_n)$ converges to
$$\theta\cos(r\pi) + (\xi \times \theta)\sin(r\pi)+\xi (\xi \cdot \theta)(1-\cos(r\pi))=R_{\xi,r}(\theta).$$

By the continuity of $h_K$ and $h_L$, the function $h_L(R_{\xi_n,r}(\theta_n))$ converges to $h_L(R_{\xi,r}(\theta))$ and $h_K(\theta_n)$ converges to $h_K(\theta).$ It follows that $h_L(R_{\xi,r}(\theta)) = h_K(\theta)$ for every $\theta \in \xi^{\perp}$. This means that $K_{|\xi^{\perp}}$ rotated by $r\pi$ coincides with $L_{|\xi^{\perp}}$, and so $\xi \in F_r$.

\end{proof}

Two-dimensional bodies of constant width play an important role in our analysis. Define the set $\Sigma \subset S^2$ by $ \Sigma =\{ \xi \in S^2: K_{|\xi^{\perp}}$ has constant width$\}$. If $\xi_1,\xi_2 \in \Sigma$, then $\xi_1^{\perp} \cap S^2$ and $\xi_2^{\perp} \cap S^2$ must intersect, which implies that $K_{|\xi_1^{\perp}}$ and $K_{|\xi_2^{\perp}}$ must have the same width, which will be denoted $M$.

\begin{lemma} \label{lem: Sigma}

$\Sigma$ is closed.

\end{lemma}

\begin{proof}

Let $\{ \xi_n \}_{n=1}^{\infty} \subset \Sigma$ be a sequence such that $\xi_n$ converges to $\xi \in S^2$, and let $\theta \in \xi^{\perp} \cap S^2$ be arbitrary. For each $n,$ there is some $\theta_n \in \xi_n^{\perp}$ so that $\theta_n$ converges to $\theta.$ Since the width function is continuous, $width_K(\theta_n)$ converges to $width_K(\theta)$, but $width_K(\theta_n)=M$ for all $M$. Therefore, $width_K(\theta)=M$ for all $\theta \in \xi^{\perp}\cap S^2.$

\end{proof}

The next lemma is used during the proof of the central lemma in the next section.

\begin{lemma} \label{lem: Baire}
Let $\theta \in S^2$ and $\epsilon > 0$, and suppose there exists a countable collection of closed subsets \{$F_n\}$ of $S^2$ with $\overline{S(\theta, \epsilon)}= \cup_{n=1}^{\infty}F_n$. Then there exists some $n \in \mathbb{N}$ where $\text{int}(F_n)$ is nonempty.
\end{lemma}

\begin{proof}

The set $\overline{S(\theta, \epsilon)}$ is compact in $S^2$ and therefore is a complete metric space. The Baire category theorem (see for example \cite{Rudin} page 98) implies that if $\overline{S(\theta, \epsilon)}= \cup_{n=1}^{\infty}F_n$, there is some $n \in \mathbb{N}$ with $\text{int}( \overline{F_n})=\text{int}(F_n) \neq \emptyset$.

\end{proof}

\section{Main Results}

Our strategy in this section will be to combine information about the original bodies and the corresponding dual bodies to reduce the problem to a statement about a system of equations. Solving this system will then reduce further to solving a quartic polynomial equation in one variable. Using this, we will see that any projection which can be rotated by an amount between 0 and $\pi$ into the other projection must be a disk. What this shows is that there actually are no rotations except for the zero rotation and a reflection about the origin.

The following lemma is due to Golubyatnikov (\cite{Golubyatnikov} page 17), and is essential to the proof of the main theorem of this paper.

\begin{lemma} \label{lem: Gol}
If the projections of two convex bodies $K,L \subset \mathbb{R}^3$ are all SO(2) congruent, then $S^2= F_0 \cup F_1 \cup \Sigma$.
\end{lemma}

We will prove a lemma analogous to Lemma \ref{lem: Gol} which is concerned with the dual bodies $K^*$ and $L^*$ of $K$ and $L$. The key is that duality takes projections into sections: $(K_{|\xi^{\perp}})^* =K^* \cap \xi^{\perp}$ for any $\xi \in S^2$ (\cite{Gardner} page 22). This is used in the proof of the following lemma which shows that rotational congruence is inherited by sections of the dual bodies.

\begin{lemma} \label{lem: rotation sections}
For all $r \in [0,1],$ $F_r=\{ \xi \in S^2 : K^* \cap \xi^{\perp}$ rotated by $ r\pi$ is $L^* \cap \xi^{\perp}\}.$
\end{lemma}

\begin{proof}

If $\xi \in F_r$, let $\Phi_r: \xi^{\perp} \mapsto \xi^{\perp}$ be the rotation of $\xi^{\perp}$ about the origin by $r\pi.$ Since $\Phi_r$ is a 1-1 linear transformation  such that $\Phi_r(K_{|\xi^{\perp}})=L_{|\xi^{\perp}}$, we have $(\Phi_r(K_{|\xi^{\perp}}))^*=(L_{|\xi^{\perp}})^*=L^* \cap \xi^{\perp}$. Also, $(\Phi_r(K_{|\xi^{\perp}}))^* = (\Phi^{-1}_r)^t(K_{|\xi^{\perp}})^*$  (\cite{Gardner} page 21) and $(\Phi^{-1}_r)^t=\Phi_r$, which implies that $\Phi_r(K^* \cap \xi^{\perp})=L^* \cap \xi^{\perp}$.

\end{proof}

Define the function $\tau_{K^*}: S^2 \mapsto \mathbb{R}$ by $\tau_{K^*}(\xi)=(\rho^2_{K^*}(\xi)+\rho^2_{K^*}(-\xi))/2$, and define $\tau_{L^*}$ similarly. Lemma \ref{lem: Gol} is a statement about projections and the width function of $K$, whereas he function $\tau_{K^*}$ contains information about the sections of $K^*$. We can define the set $\Lambda \subset S^2$ analogously to $\Sigma$ by $\Lambda = \{ \xi \in S^2 : \tau_{K^*}$ restricted to $ \xi^{\perp} \cap S^2$ is constant$\}.$ Observe that if $\xi \in F_r,$ since the radial function measures distance from the origin, it follows from Lemma \ref{lem: rotation sections} that $\tau_{K^*}(\theta) = \tau_{L^*}(R_{\xi,r}(\theta))$ for every $\theta \in \xi^{\perp} \cap S^2.$

\begin{lemma} \label{lem: tau}
$\tau_{K^*} (\xi)=\tau_{L^*}(\xi)$ for every $\xi \in S^2.$
\end{lemma}

\begin{proof}

Since all sections of $K^*$ are congruent to corresponding sections of $L^*$, $\text{area}(K^* \cap \xi^{\perp}) = \text{area}(L^* \cap \xi^{\perp})$ for every unit vector $\xi.$ Since the area of the section can be expressed as $\frac{1}{2} \int_0^{2\pi} \rho^2_{K^* \cap \xi^{\perp}}(\theta) d\theta,$ we can conclude

$$\int_{\xi^{\perp} \cap S^2}\frac{\rho^2_{K^*}(\theta)+\rho^2_{K^*}(-\theta)}{2} d\theta=\int_{\xi^{\perp} \cap S^2}\frac{\rho^2_{L^*}(\theta)+\rho^2_{L^*}(-\theta)}{2} d\theta$$ for every $\xi \in S^2.$ This can be rewritten as

$$\int_{\xi^{\perp} \cap S^2} \tau_{K^*}(\theta)-\tau_{L^*}(\theta) d\theta =0$$ for all unit vectors $\xi$. Thus, the spherical Radon transform of the even function $\tau_{K^*}-\tau_{L^*}$ is identically zero on $S^2$, and so (\cite{Gardner} page 430) implies that $\tau_{K^*}$ and $\tau_{L^*}$ coincide everywhere.

\end{proof}

\begin{lemma} \label{lem: Mod Gol}
Let $K,L \subset \mathbb{R}^3$ be convex bodies containing the origin as an interior point so that for all $\xi \in S^2$, $K_{|\xi^{\perp}}$ can be rotated about the origin into $L_{|\xi^{\perp}}$ . Then $S^2= F_0 \cup F_1 \cup \Lambda$.
\end{lemma}

The proof of this lemma closely resembles Golubyatnikov's proof starting on page 17 of \cite{Golubyatnikov}, and will be postponed until the end of the paper. If we assume that this lemma has been proven, we can complete the proof of Theorem \ref{thm: Main}. From Lemma \ref{lem: Mod Gol}  and Lemma \ref{lem: Gol}, we know that if $\xi \in S^2$ is not in $F_0 \cup F_1$, then it must be that $\xi \in \Sigma \cap \Lambda$. The main idea is to use duality to show that if $\xi \in \Sigma \cap \Lambda$, then $\xi \in F_0 \cup F_1$, and therefore $S^2 = F_0 \cup F_1$. Golubyatnikov has  proven that this then implies that $S^2= F_0$ or $S^2=F_1$ (\cite{Golubyatnikov} page 22), which completes the proof of Theorem \ref{thm: Main}. All that remains is to prove the following corollary.

\begin{corollary}

$S^2= F_0 \cup F_1.$

\end{corollary}

\begin{proof}

By Lemma \ref{lem: Gol} and Lemma \ref{lem: Mod Gol}, $S^2 \backslash (F_0 \cup F_1)$ is contained in $\Sigma \cap \Lambda$, so it suffices to prove that $\Sigma \cap \Lambda$ is a subset of $F_0 \cup F_1$. If $\xi \in \Sigma \cap \Lambda,$ Lemma \ref{lem: Gol} implies that there exists a constant $a \in \mathbb{R}$ so that for all $\theta \in \xi^{\perp}$, $$h_K(\theta) + h_K(-\theta)=a.$$ By Lemma \ref{lem: Mod Gol}, there is a constant $b \in \mathbb{R}$ so that for all $\theta \in \xi^{\perp}$, $$\rho_{K^*}^2(\theta) + \rho_{K^*}^2(-\theta)=b.$$ Since $h_K(\theta)=1/\rho_{K^*}(\theta)$ for any unit vector $\theta$ (\cite{Gardner} page 20), the second equation can be rewritten as $$h_K(\theta)^{-2} + h_K(-\theta)^{-2}=b$$ for all $\theta \in \xi^{\perp}$.

Consider the system of equations $x+y=a$ and $ x^{-2}+y^{-2}=b.$ Since the origin is an interior point of $K$, we can assume that neither $x$ nor $y=a-x$ is equal to zero, and thus this system can be expressed as the quartic equation $$(x-a)^2+ x^2=bx^2(a-x)^2,$$ which has at most four real valued solutions in $x$. For each $\theta \in \xi^{\perp},$ $h_K(\theta)$ is a solution to this equation. Since $h_K$ is a continuous function of $\theta,$ the intermediate value theorem implies that $h_K$ is constant on $\xi^{\perp}.$ This implies that both $K_{|\xi^{\perp}}$ and $L_{|\xi^{\perp}}$ are disks, and thus $K_{|\xi^{\perp}}=L_{|\xi^{\perp}}.$ It follows that $\xi \in F_0,$ and therefore $\Sigma \cap \Lambda \subset F_0 \cup F_1$.

\end{proof}

We conclude this paper by returning to the proof of Lemma \ref{lem: Gol}.

\begin{proof} (Lemma \ref{lem: Mod Gol}) 

To prove the lemma, we will show that the set $F = S^2 \backslash (F_0 \cup F_1 \cup \Lambda)$ is empty. An argument similar to the proof of Lemma \ref{lem: Sigma} (just replace $width_K$ with $\tau_{K^*}$) shows that $\Lambda$ is closed, so it follows from Lemma \ref{lem: closed} and Lemma \ref{lem: rotation sections} that $F$ is an open set. Suppose there is a unit vector $\xi \in F$ such that $\xi \in F_r$ for some irrational $r \in (0,1)$. By Lemma \ref{lem: rotation sections} and Lemma \ref{lem: tau}, we can conclude that for every $\theta \in \xi^{\perp}\cap S^2$, 
$$\tau_{K^*}(\theta)=\tau_{L^*}(R_{\xi,r}(\theta))=\tau_{K^*}(R_{\xi,r}(\theta)).$$ Fixing some $\theta_0 \in \xi^{\perp} \cap S^2$, an inductive argument shows that $$\tau_{K^*}(R_{\xi, nr}(\theta_0))=\tau_{K^*}(\theta_0) \forall n \in \mathbb{N}.$$ Since $r$ is irrational, the set $\{ R_{\xi, nr}(\theta_0) : n \in \mathbb{N} \}$ is dense in the circle $\xi^{\perp} \cap S^2$. The function $\tau_{K^*}$ is continuous on $S^2$, and it takes on a single value on a dense subset of $\xi^{\perp} \cap S^2$, so it follows that $\tau_{K^*}$ takes a single value on $\xi^{\perp} \cap S^2$. Therefore $K^* \cap \xi^{\perp}$ is a disk, and thus so is $L^* \cap \xi^{\perp}$, which contradicts the assumption that $\xi \notin \Lambda.$

If $F$ is nonempty and none of the sections orthogonal to elements of $F$ coincide after some irrational angle, then $F$ can be rewritten as $F=\cup_r (F \cap F_r)$ for some subset of the rational numbers contained in $(0,1)$. We claim that there exists some  rational $r_0 \in (0,1)$ with $\text{int}(F \cap F_{r_0}) \neq \emptyset$. To see this, since $F$ is open and assumed nonempty, there exists $\theta \in F$ and $\epsilon > 0$ with $\overline{ S(\theta, \epsilon)} \subset F$. Then $\overline{S(\theta, \epsilon)} = \cup_r \overline{S(\theta, \epsilon)}\cap F_r$, and Lemma \ref{lem: Baire} implies there exists some rational $r_0 \in (0,1)$ with $\emptyset \neq \text{int}(\overline{S(\theta, \epsilon)} \cap F_{r_0}) \subset \text{int}(F \cap F_{r_0})$.

 Fix some $\xi$ and $\epsilon > 0$ with the spherical disk $S(\xi, \epsilon)$ contained in $F \cap F_{r_0}.$ Then the continuous function $\tau_{K^*}$ is not constant along $\xi^{\perp} \cap S^2$. Therefore,  infinitely many values $c$ exist with corresponding unit vectors $w_c \in \xi^{\perp} \cap S^2$ so that $\tau_{K^*}(w_c)=c.$ The rest of the proof will be spent using these values to construct marks on the sphere which are geometrically impossible.

For each value $c = \tau_{K^*}(w_c)$ for some $w_c \in \xi^{\perp} \cap S^2$, denote by $w'_c$ the vector obtained by rotating $w_c$ by $r_0 \pi$ along $\xi^{\perp} \cap S^2.$ We claim there is an open arc $l^1_c \subset S(w_c, r_0)$ containing $w'_c$ with $\tau_{K^*}$ identically equal to $c$ on $l^1_c$. We will then construct another arc $l^3_c$ on $S^2$ which intersects $l^1_c$ at $w'_c$ on which $\tau_{K^*}$ is also constant, and we will define $X_c =l^1_c \cup l^3_c$  (see Figure 1).

\begin{figure}[ht!] 
\centering
\includegraphics[width=50mm]{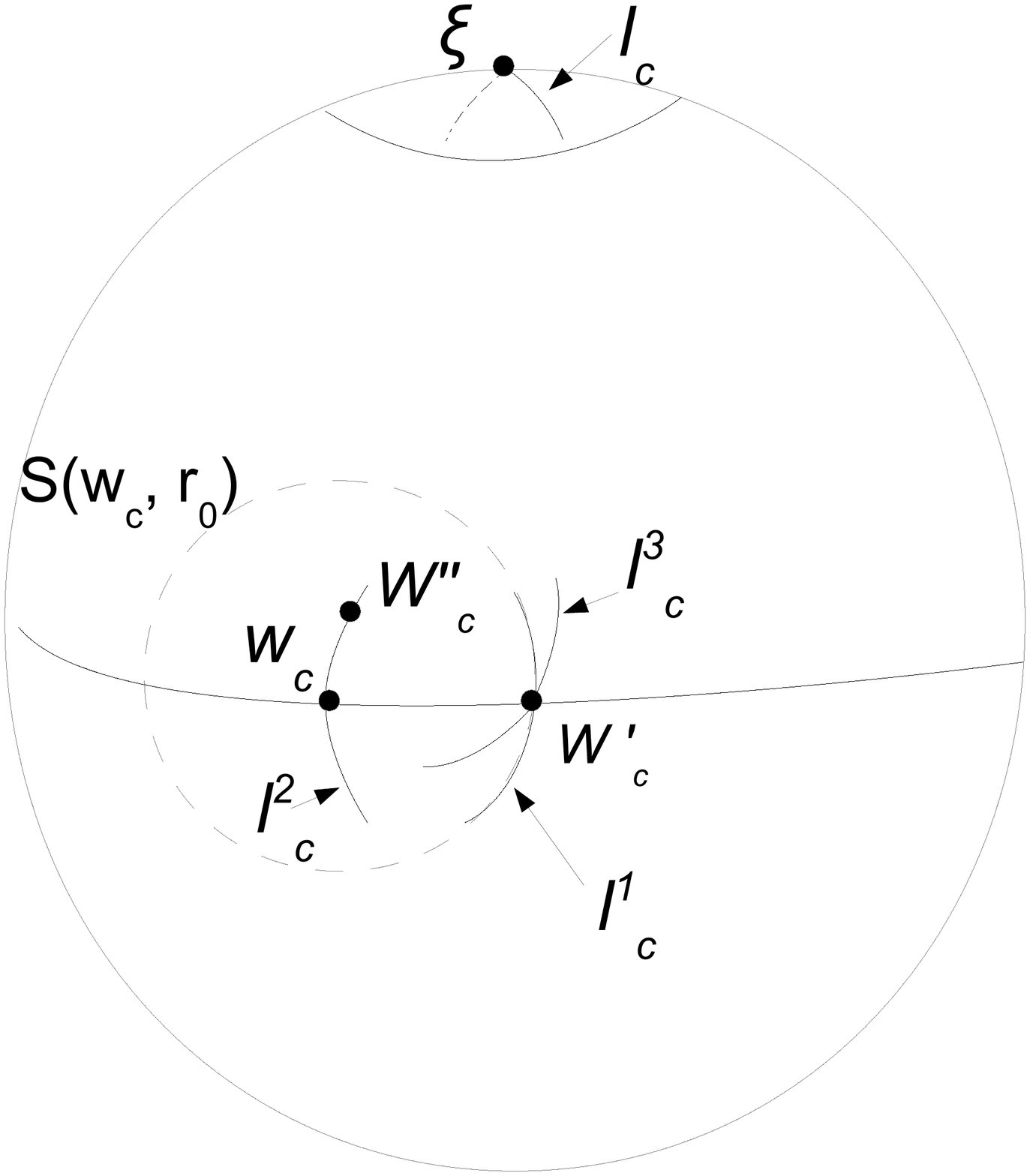}
\caption{The construction of $l_c^1 \cup l_c^3$ for a clockwise rotation by an angle between 0 and $2\pi$ }
\label{overflow}
\end{figure}

Since the spherical disk $S(\xi, \epsilon)$ is contained in $F_{r_0}$ and $\xi \in w_c^{\perp} \cap S^2$, there exists an open arc $l_c \subset w_c^{\perp} \cap S^2$ centered at $\xi$ and contained in $S(\xi, \epsilon) \subset F_{r_0}$. For any $v \in l_c$, $v$ is orthogonal to $w_c$, and so $w_c \in v^{\perp}$, which implies that the great circle $v^{\perp} \cap S^2$ intersects $S(w_c, r_0)$ at the point $R_{v, r_0}(w_c).$ Call $l^1_c = \{ R_{v,r_0}(w_c) : v \in l_c \}$. If $\theta \in l^1_c$ with $\theta = R_{v,r_0}(w_c)$, where $v \in l_c$, the fact that $l_c \subset F_{r_0}$ and Lemma \ref{lem: tau} imply that,

$$c=\tau_{K^*}(w_c)= \tau_{L^*}(R_{v,r_0}(\theta)) =\tau_{L^*}(\theta)=\tau_{K^*}(\theta)$$ which proves the claim.

We now construct the second arc $l^3_c$. If we start at $w'_c$ and rotate the projections of $L^*$ in the reverse direction, by a similar argument applied to $L^*$ we can create $S(w'_c, r_0)$ and an open arc $l^2_c \subset S(w'_c, r_0)$ so that $\tau_{K^*}$ takes only the value $c$ on $l^2_c$. Next, we can pick a third unit vector $w''_c \in l^2_c$ distinct from $w_c$, and consider the circle $S(w''_c, r_0)$. Using an similar argument (create an arc in the spherical disk $S(\xi, \epsilon)$ centered at the preimage of $w''_c$ and consider the image of this arc under the rotation), we can create an arc $l^3_c \subset S(w''_c, r_0)$ on which $\tau_{K^*}$ takes only the value $c$. If we define $X_c= l^1_c \cup l^3_c$ to by the cross mark formed by the two arcs, we see that the function $\tau_{K^*}$ takes only the value $c$ on $X_c$.

This construction can be done identically for every value taken by $\tau_{K^*}$ on $\xi^{\perp} \cap S^2$, so we have constructed an infinite family of congruent ``X" figures on the unit sphere. For distinct values $c_1, c_2$ taken by $\tau_{K^*}$ on $\xi^{\perp} \cap S^2$, it follows from the construction that $X_{c_1}$ and $X_{c_2}$ are disjoint. Since it is impossible to construct infinitely many congruent mutually disjoint ``X" figures on the sphere, this completes the proof.

\end{proof}

\end{document}